\documentclass[11pt]{amsart}
\usepackage{amsthm, amsmath, amssymb}
\usepackage{mathtools}
\usepackage{hyperref}
\usepackage[top=1.25in, bottom=1.25in,left=1.25in,right=1.25in]{geometry}
\usepackage{cleveref}
\usepackage{etoolbox}

\newtheorem{thm}{Theorem}[section]
\newtheorem{lem}[thm]{Lemma}

\newtheorem{cor}[thm]{Corollary}

\theoremstyle{definition}
\newtheorem{defin}[thm]{Definition}

\newtheorem{claim}[thm]{Claim}

\theoremstyle{remark}
\newtheorem{remark}[thm]{Remark}

\crefname{thm}{Theorem}{Theorems}
\crefname{lem}{Lemma}{Lemmas}
\crefname{prop}{Proposition}{Propositions}
\crefname{cor}{Corollary}{Corollaries}
\crefname{conj}{Conjecture}{Conjectures}
\crefname{defin}{Definition}{Definitions}
\crefname{quest}{Question}{Questions}
\crefname{example}{Example}{Examples}
\crefname{prob}{Problem}{Problems}
\crefname{claim}{Claim}{Claims}
\crefname{remark}{Remark}{Remarks}

\AtBeginEnvironment{lem}{\crefalias{thm}{lem}}
\AtBeginEnvironment{prop}{\crefalias{thm}{prop}}
\AtBeginEnvironment{cor}{\crefalias{thm}{cor}}
\AtBeginEnvironment{conj}{\crefalias{thm}{conj}}
\AtBeginEnvironment{defin}{\crefalias{thm}{defin}}
\AtBeginEnvironment{quest}{\crefalias{thm}{quest}}
\AtBeginEnvironment{example}{\crefalias{thm}{example}}
\AtBeginEnvironment{prob}{\crefalias{thm}{prob}}
\AtBeginEnvironment{claim}{\crefalias{thm}{claim}}
\AtBeginEnvironment{remark}{\crefalias{thm}{remark}}

\DeclareTextFontCommand{\emph}{\color{blue}\em}
\usepackage{tikz}

\DeclareMathOperator{\Supp}{Supp}
\DeclareMathOperator{\lr}{lr}
\DeclareMathOperator{\lb}{lb}
\DeclareMathOperator{\br}{br}

\title{Separable elements and Bruhat orders}

\author{Christian Gaetz}
\address[Gaetz]{Department of Mathematics, University of California, Berkeley, CA, USA.}
\email{\href{mailto:gaetz@berkeley.edu}{{\tt gaetz@berkeley.edu}}}
\author{Yibo Gao}
\address[Gao]{Beijing International Center for Mathematical Research, Peking University, Beijing, China.}
\email{\href{mailto:gaoyibo@bicmr.pku.edu.cn}{{\tt gaoyibo@bicmr.pku.edu.cn}}}

\date{\today}

\begin{document}

\begin{abstract}
We prove the curious fact that, in any Coxeter group, the intersection of a principal left weak order ideal $[e,x]_L$ with a principal right weak order ideal $[e,y]_R$ has a unique maximal element $\lr(x,y)$ under \textit{Bruhat} order. We apply this to resolve a conjecture of Hart and Rowley on involutions in weak order ideals.

This fact arises from, and is applied to, our study of the theory of \textit{separable elements} in finite Weyl groups. We use it to prove that for $U=[e,u]_R$ in a Coxeter group $W$, the multiplication map $W/U\times U\rightarrow W$ is surjective, where $W/U$ is the Bj\"{o}rner--Wachs generalized quotient. This resolves and generalizes a conjecture of the authors. For finite Weyl groups, we further prove that this map is bijective if and only if $u$ is separable, thereby classifying the \textit{splittings} of finite Weyl groups into weak order intervals, and resolving the natural extension of a question of Bj\"orner--Wachs.  
\end{abstract}

\maketitle

\section{Introduction} \label{sec:intro}
This paper is the third in a sequence \cite{separable-2,separable-1} developing the theory of \textit{separable elements} in finite Weyl groups $W$. In \cite{separable-1}, we generalize the well-studied class of separable permutations to arbitrary finite Weyl groups and develop their basic properties, including the rank generating functions of their lower intervals in weak order and a Billey--Postnikov pattern criterion for recognizing them. Separable elements were subsequently shown to be distinguished by their actions on the canonical basis \cite{Gossow-Yacobi}. 

For $U \subseteq W$, Bj\"{o}rner and Wachs \cite{Bjorner-Wachs} define the \emph{generalized quotient} $W/U$: 
\[W/U\coloneqq \{w\in W\mid \ell(wu)=\ell(w)+\ell(u)\text{ for all }u\in U\}.\]
Subsets $X,Y$ form a \emph{splitting} of $W$ if multiplication is a length-additive bijection $X \times Y \to W$. In \cite{separable-2} it was shown that the splittings of the symmetric group $S_n$ are exactly the pairs $(S_n/[e,u]_R, [e,u]_R)$ with $u$ separable, resolving a question from \cite{Bjorner-Wachs}. We also showed that $S_n/[e,u]_R \times [e,u]_R \to S_n$ is \textit{surjective} even for non-separable $u$. 

In the present paper, we extend these two main theorems to all finite Weyl groups, resolving conjectures from \cite{separable-2}. A key new idea is the following result relating the left weak, right weak, and Bruhat orders on any Coxeter group.

\begin{thm}\label{thm:left-right-max}
Let $W$ be any Coxeter group and $x,y\in W$. Then $[e,x]_L\cap [e,y]_R$ contains a unique maximal element, denoted $\lr(x,y)$, under Bruhat order.
\end{thm}

\noindent This \emph{left-right meet} operation on a Coxeter group strikes us as quite a  natural one, and we hope it will find many connections and applications to Coxeter group theory. We give several applications below.

Hart and Rowley conjectured \Cref{cor:Hart-Rowley} for finite Coxeter groups \cite[Conj.~1.2]{Hart-Rowley}; we settle their conjecture as well as its generalization to arbitrary Coxeter groups.

\begin{cor}[Conj.~1.2 of \cite{Hart-Rowley}]\label{cor:Hart-Rowley}
Let $W$ be a Coxeter group and let $w\in W\setminus\{e\}$. Then $[e,w]_R$ contains a unique involution $a(w)$ of maximal length. Moreover, every involution $\iota\leq_R w$ satisfies $\iota\leq_B a(w)$.
\end{cor}

\Cref{thm:left-right-max} is also instrumental in our proof of the surjectivity theorem for general Coxeter groups. Surjectivity was previously known only in finite types $A$ \cite{separable-2} and $B$ \cite{separable-typeB}.

\begin{thm}[Conj.~1 of \cite{separable-2}]\label{thm:surjection}
Let $W$ be any Coxeter group, and let $u \in W$ and $U=[e,u]_R$. Then the multiplication map $W/U\times U\rightarrow W$ is a surjection.
\end{thm}

\noindent Despite the simple statement of \Cref{thm:surjection}, there is a rich history behind it. The numerical inequality implied by \Cref{thm:surjection} in type $A$ can be interpreted naturally in terms of linear extensions of posets, and was first proved in that setting by Sidorenko \cite{Sidorenko}. Interest in Sidorenko's inequality has continued with further proofs by Bollob\'as--Brightwell--Sidorenko \cite{BBS} using a special case of the Mahler conjecture from convex geometry together with Lov\'asz's (weak) Perfect Graph Theorem, and more recently by the authors in \cite{hull-metric} and by Chan--Pak--Panova \cite{CPP}. None of these proofs implies the group-theoretic statement in \Cref{thm:surjection}. \Cref{thm:surjection} also proves the conjectured \textit{hull metric inequality} of \cite[Conj.~1.2]{hull-metric} for pairs of elements whose convex hull is all of $W$.

We also extend the classification of splittings into weak order intervals\footnote{It is noted in \cite{separable-2} that the Weyl group of $F_4$ has a splitting $(X,Y)$ where $Y$ is not an interval in weak order. Such splittings are ruled out in the type $A$ result in \textit{op. cit.} but not by \Cref{thm:separable-bijection}.} from types $A$ and $B$ to all finite Weyl groups. This answers the natural extension of the question of Bj\"{o}rner--Wachs.

\begin{thm}[Conj.~2 of \cite{separable-2}]\label{thm:separable-bijection}
Let $U=[e,u]_R$ in a finite Weyl group $W$. Then the multiplication map $W/U\times U\rightarrow W$ is a bijection, i.e., $(W/U,U)$ is a splitting of $W$, if and only if $u$ is separable.
\end{thm}

Finally, we characterize separable elements directly in terms of the left-right meet.

\begin{cor}
\label{cor:separable-iff-LR-e}
An element $u$ in a finite Weyl group $W$ is separable if and only if
\[
\lr(w_0u^{-1},u)=e.
\]
\end{cor}

\subsection*{Outline} \Cref{sec:prelim} covers background on Coxeter groups, the weak and Bruhat orders on them, generalized quotients and splittings, and separable elements. \Cref{sec:left-right} proves \Cref{thm:left-right-max} and uses it to prove the conjecture of Hart--Rowley in \Cref{cor:Hart-Rowley}. \Cref{sec:surjection} then applies \Cref{thm:left-right-max} together with several technical lemmas to prove \Cref{thm:surjection}. Finally, \Cref{sec:non-sep} proves the remaining direction of \Cref{thm:separable-bijection}, that non-separable elements do not induce splittings of $W$.

\section{Preliminaries}\label{sec:prelim}
\subsection{Coxeter groups and Bruhat orders}
We follow \cite{Bjorner-Brenti} for background and notation on Coxeter groups. We write $e$ for the identity element of a Coxeter system $(W,S)$. An expression $w=s_1\cdots s_m$, with $s_i\in S$, is \emph{reduced} if $m$ is as small as possible, and in this case $\ell(w) \coloneqq m$ is the \emph{length} of $w$. More generally, we say a product $w_1\cdots w_r$ is \emph{reduced} if $\ell(w_1\cdots w_r)=\ell(w_1)+\cdots+\ell(w_r)$. 

The left and right \emph{weak (Bruhat) orders} are defined by
\begin{align*}
u\leq_L w
&\quad\Longleftrightarrow\quad
\ell(w)=\ell(wu^{-1})+\ell(u), \\
u\leq_R w
&\quad\Longleftrightarrow\quad
\ell(w)=\ell(u)+\ell(u^{-1}w).
\end{align*}
Thus $u\leq_L w$ means that a reduced expression for $u$ occurs as a suffix of a reduced expression for $w$, while $u\leq_R w$ means that it occurs as a prefix. The closed intervals in these orders are denoted $[u,w]_L$ and $[u,w]_R$, respectively.

Let $T=\{wsw^{-1}\mid w\in W,\ s\in S\}$ be the set of \emph{reflections} in $W$. We write $T_L(w) \coloneqq \{t \in T \mid \ell(tw)<\ell(w)\}$ and $T_R(w) \coloneqq \{t \in T \mid \ell(wt)<\ell(w)\}$. It is known that $u \leq_L w$ (resp. $u \leq_R w$) if and only if $T_R(u) \subseteq T_R(w)$ (resp. $T_L(u) \subseteq T_L(w)$). The right and left \emph{descent} sets are $D_R(w)=T_R(w) \cap S$ and $D_L(w)=T_L(w) \cap S$, respectively.

The \emph{(strong) Bruhat order} is the partial order generated by the covering relations $w\lessdot_B wt$ for $t\in T$ with $\ell(wt)=\ell(w)+1$. Its closed intervals are denoted $[u,w]_B$. The \emph{subword property} says that $u\leq_B w$ if and only if some (equivalently, any) reduced expression for $w$ has a subword which is a reduced expression for $u$. The \emph{exchange condition} says that if $w=s_1\cdots s_m$ is reduced and $s\in D_R(w)$, then $ws=s_1\cdots\widehat{s_i}\cdots s_m$ for some $i$, and the resulting expression is reduced. The \emph{lifting property} says that if $u\leq_B w$ and $s\in D_R(w)$, then $us\leq_B ws$ when $s\in D_R(u)$, and $u\leq_B ws$ when $s\notin D_R(u)$. 

Bruhat order refines both weak orders, and all three orders are graded by length. Inversion preserves Bruhat order and interchanges the two weak orders.

The \emph{support} $\Supp(w)$ is the set of simple reflections occurring in some (equivalently, any) reduced expression for $w$. Equivalently, $\Supp(w)=\{s\in S\mid s\leq_B w\}$.

For $K\subseteq S$, let $W(K)$ be the corresponding \emph{standard parabolic subgroup}. The length function and the three orders on $W(K)$ agree with their restrictions from $W$. Every $w\in W$ admits a unique \emph{parabolic decomposition} $w=w^K w_K$ that is length-additive, where $w^K\in W^K\coloneqq \{w\in W\mid D_R(w)\cap K=\emptyset\}$, $w_K\in W(K)$. We will also make use of the left parabolic decomposition $w={}_K w\cdot {}^Kw$, where ${}_K w\in W(K)$ and ${}^Kw\in\{w\in W\mid D_L(w)\cap K=\emptyset\}$. 

When $W$ is finite, it has a unique element $w_0$ of maximum length, called the \emph{longest element}. It satisfies $w_0^2=e$ and $\ell(w_0w)=\ell(ww_0)=\ell(w_0)-\ell(w)$ for every $w\in W$. Left and right multiplication by $w_0$ reverse each of the three orders. We write $w_0(K)$ for the longest element of $W(K)$.

The \emph{generalized quotient} associated to a subset $U\subseteq W$ is \[
W/U \coloneqq \{w\in W\mid \ell(wu)=\ell(w)+\ell(u)\text{ for all }u\in U\}.
\] 
A pair of subsets $(X,Y)$ forms a \emph{splitting} of $W$ if the multiplication map $X\times Y\rightarrow W$ is a length-additive bijection. Parabolic decomposition gives the basic example $(W^K,W(K))$, with $W^K=W/W(K)$. If $W$ is finite and $U=[e,u]_R$, then $W/U=[e,w_0u^{-1}]_L$ \cite{Bjorner-Wachs}.

\subsection{Separable elements in finite Weyl groups}\label{sub:separable}

In this subsection, let $\Phi$ be a finite crystallographic root system with a chosen set of simple roots $\Delta=\{\alpha_s\mid s\in S\}$ and corresponding positive roots $\Phi^+$. Its Weyl group is $W=W(\Phi)$, with simple reflections $S$. Write $\Phi^-=-\Phi^+$, and use $\beta>0$ and $\beta<0$ to mean $\beta\in\Phi^+$ and $\beta\in\Phi^-$, respectively. For $\beta\in\Phi$, let $t_\beta$ denote the reflection orthogonal to $\beta$.

For $\alpha\in\Delta$ and $\beta\in\Phi$, write $[\alpha]\beta$ for the coefficient of $\alpha$ in the simple-root expansion of $\beta$. For $\beta\in\Phi^+$, define its \emph{height} by $\operatorname{ht}(\beta)=\sum_{\alpha\in\Delta}[\alpha]\beta$ and its support by $\Supp(\beta)=\{s\in S\mid [\alpha_s]\beta>0\}$. The \emph{root poset} is the partial order on $\Phi^+$ defined by $\beta\leq\gamma$ if $\gamma-\beta$ is a nonnegative linear combination of simple roots. In particular, $\alpha_s\leq\beta$ if and only if $s\in\Supp(\beta)$.

The right and left \emph{inversion sets} of $w\in W$ are
\[
I_R(w)=\{\beta\in\Phi^+\mid w\beta<0\},
\qquad
I_L(w)=I_R(w^{-1}).
\]
An element is uniquely determined by its inversion set, and a subset $A\subseteq\Phi^+$ is the (left or right) inversion set of a group element if and only if it is \emph{biconvex}:
\begin{itemize}
\item if $\beta,\gamma\in A$ and $\beta+\gamma\in\Phi^+$, then $\beta+\gamma\in A$; and
\item if $\beta,\gamma\notin A$ and $\beta+\gamma\in\Phi^+$, then $\beta+\gamma\notin A$.
\end{itemize}
For a subspace $E'$ of the ambient vector space spanned by $\Phi$, we obtain a root subsystem $\Phi'=E'\cap\Phi$, with $(\Phi')^+=E'\cap\Phi^+$. For $w\in W$, define its restriction $w|_{\Phi'}$ to be the unique element of $W(\Phi')$ such that $I_R(w|_{\Phi'})=I_R(w)\cap\Phi'$. For a standard parabolic subsystem $\Phi(K)$, we simply have $w|_{\Phi(K)}=w_K$. This map was first used by Billey and Postnikov \cite{Billey-Postnikov} to generalize pattern avoidance to finite Weyl groups and to study (rational) smoothness of Schubert varieties.

The weak orders are characterized by
\[
u\leq_L w\quad\Longleftrightarrow\quad I_R(u)\subseteq I_R(w),
\qquad
u\leq_R w\quad\Longleftrightarrow\quad I_L(u)\subseteq I_L(w).
\]
We have $T_L(w)=\{t_{\beta} \mid \beta \in I_L(w)\}$ and similarly for $T_R$. If $w=s_1\cdots s_m$ is reduced, then
\begin{align*}
I_L(w)
&=\{s_1\cdots s_{i-1}(\alpha_{s_i})\mid 1\leq i\leq m\},\\
I_R(w)
&=\{s_m\cdots s_{i+1}(\alpha_{s_i})\mid 1\leq i\leq m\}.
\end{align*}

We now recall the recursive definition of separability.
\begin{defin}[\cite{separable-1}]\label{def:separable}
We say that $w\in W$ is \emph{separable} if one of the following holds:
\begin{enumerate}
\item[(S1)] $\Phi$ has rank at most one.
\item[(S2)] $\Phi=\Phi_1\oplus\cdots\oplus\Phi_r$ is reducible, and the restriction $w|_{\Phi_i}$ is separable for every $i$.
\item[(S3)] $\Phi$ is irreducible, and there is a simple root $\alpha_s\in\Delta$ such that $w_{S\setminus\{s\}}$ is separable and either
\[
\{\beta\in\Phi^+\mid \alpha_s\leq\beta\}\subseteq I_R(w)
\quad\text{or}\quad
\{\beta\in\Phi^+\mid \alpha_s\leq\beta\}\cap I_R(w)=\emptyset.
\]
\end{enumerate}
In (S3), we call $\alpha_s$ a \emph{full pivot} in the first case and an \emph{empty pivot} in the second. We also refer to the corresponding simple reflection or Dynkin vertex as a pivot.
\end{defin}
By \cite[Prop.~5.2]{separable-1}, every standard parabolic restriction of a separable element is separable. Separability is also preserved by inversion and multiplication on either side by $w_0$ \cite[Cor.~2]{separable-2}. In particular, if $xy=w_0$, then $x$ is separable if and only if $y$ is separable.

Finally, if $u$ is separable and $U=[e,u]_R$, then $(W/U,U)$ is a splitting of $W$ by \cite[Thm.~3]{separable-2}. Thus the separable direction of \Cref{thm:separable-bijection} is already known; the remaining direction is the subject of the non-separable cases in \Cref{sec:non-sep}.

\section{Intersecting left and right weak order intervals}\label{sec:left-right}
We now prove \Cref{thm:left-right-max}. Let $W$ be any Coxeter group, and for $x,y\in W$, consider the following sets of elements:
\begin{align*}
\mathcal{LB}(x,y)&\coloneqq \{w\in W\mid w\leq_L x\text{ and }w\leq_B y\},\\
\mathcal{BR}(x,y)&\coloneqq \{w\in W\mid w\leq_B x\text{ and }w\leq_R y\}.
\end{align*}

\Cref{lem:intersection-left-bruhat} is known at some level of generality, see \Cref{rem:matroid}. We include a proof here for the convenience of the reader and to cover all Coxeter groups.

\begin{lem}\label{lem:intersection-left-bruhat}
The set $\mathcal{LB}(x,y)$ (resp. $\mathcal{BR}(x,y)$) contains a unique maximal element $\lb(x,y)$ (resp. $\br(x,y)$) under Bruhat order. 
\end{lem}
\begin{proof}
We first observe that complementation reverses Bruhat order on a right weak order interval: if $u,v\leq_R w$ and $u\leq_B v$, then $v^{-1}w\leq_B u^{-1}w$. To see this, write $v^{-1}w=a$, $u^{-1}w=b$ so that $w=ub=va$ with both products reduced. Induct on $\ell(v)$. The case $v=e$ is immediate. Otherwise, choose $s\in S$ with $vs<_B v$. Then $w=(vs)(sa)$ is still a reduced factorization, so $vs\leq_R w$ and $\ell(sa)=\ell(a)+1$. If $us<_B u$, then similarly $us\leq_R w$ and $\ell(sb)=\ell(b)+1$; the lifting property gives $us\leq_B vs$, so induction yields $sa\leq_B sb$. Since both $sa$ and $sb$ are reduced products, the left-handed lifting property gives $a\leq_B b$. If $us>_B u$, the lifting property instead gives $u\leq_B vs$, and induction yields $sa\leq_B b$, whence again $a\leq_B b$.

We now construct $\lb(x,y)$ directly and show that it is indeed the maximum of $\mathcal{LB}(x,y)$. For $s\in S$, define the \emph{downward Demazure operator} $\delta_s(g)=\min_{\leq_B}\{g,gs\}$. The lifting property implies that $\delta_s$ is order-preserving in Bruhat order, and clearly $\delta_s(g)\leq_B g$. Fix a reduced expression $y=s_1\cdots s_m$, set $u_m=x$, and recursively define $u_{i-1}=\delta_{s_i}(u_i)$ for $i=m,\ldots,1$. Let $i_1<\cdots<i_r$ be the indices at which $u_{i-1}\neq u_i$, and put $z=s_{i_1}\cdots s_{i_r}$ and $u=u_0$. By construction, $u=xz^{-1}$ and $\ell(u)=\ell(x)-r$. Since $\ell(z)\leq r$, the inequalities $\ell(x)\leq\ell(u)+\ell(z)\leq\ell(x)$ force equality throughout. Thus the expression $s_{i_1}\cdots s_{i_r}$ for $z$ and the factorization $x=uz$ are reduced. In particular, $z\leq_L x$, and the subword property gives $z\leq_B y$, so $z\in\mathcal{LB}(x,y)$. We claim that $z$ is the greatest element $\lb(x,y)$ of $\mathcal{LB}(x,y)$.

Now take any $w\in\mathcal{LB}(x,y)$ and choose a reduced subword of $s_1\cdots s_m$ representing $w$. Applying only the corresponding downward operators to $x$, in decreasing index order, gives $xw^{-1}$. Applying the full sequence of operators gives an element Bruhat-below this output. Therefore $u\leq_B xw^{-1}$, and because both $u$ and $xw^{-1}$ lie in $[e,x]_R$, the observation in the first paragraph of this proof gives $w=(xw^{-1})^{-1}x\leq_B u^{-1}x=z$. Hence $z$ is the greatest element of $\mathcal{LB}(x,y)$.

Dually, inversion maps $\mathcal{BR}(x,y)$ onto $\mathcal{LB}(y^{-1},x^{-1})$, giving $\br(x,y)=\lb(y^{-1},x^{-1})^{-1}$.
\end{proof}

\begin{remark}
\label{rem:matroid}
For finite Coxeter groups, \Cref{lem:intersection-left-bruhat} is \cite[Cor.~4.17]{Hivert-Schilling-Thiery} or \cite[Thm.~3]{Bump}. Moreover, the construction in the proof of \Cref{lem:intersection-left-bruhat} shows, in the language of \cite{Caselli-DAdderio-Marietti}, that every principal Bruhat interval $[e,x]_B$ is a Coxeter matroid, a special case of \cite[Thm.~4.4]{Caselli-DAdderio-Marietti}. Indeed, for $y\in W$, define the \emph{$y$-twisted Bruhat order} by $u\leq_B^y v$ if $y^{-1}u\leq_B y^{-1}v$. Since $\br(x,y)=\lb(y^{-1},x^{-1})^{-1}$, applying the construction with $(y^{-1},x^{-1})$ in place of $(x,y)$ shows that $y^{-1}\br(x,y)$ is obtained from $y^{-1}$ by applying the downward Demazure operators $\delta_{s_m},\ldots,\delta_{s_1}$, where $x^{-1}=s_1\cdots s_m$ is reduced. Now let $v\leq_B x$, and write $v^{-1}=s_{j_1}\cdots s_{j_p}$ as a subword with $j_1<\cdots<j_p$. Applying only $\delta_{s_{j_p}},\ldots,\delta_{s_{j_1}}$ to $y^{-1}$ gives an element that lies Bruhat-above $y^{-1}\br(x,y)$, as in the proof, and Bruhat-below $y^{-1}s_{j_p}\cdots s_{j_1}=y^{-1}v$, since each $\delta_s$ is order-preserving and $\delta_s(g)\leq_B gs$. Hence $y^{-1}\br(x,y)\leq_B y^{-1}v$ for every $v\leq_B x$. Thus $\br(x,y)$ is the unique minimum of $[e,x]_B$ in the $y$-twisted Bruhat order. Since this holds for every $y\in W$, it is precisely the minimality property defining a Coxeter matroid, and $y\mapsto \br(x,y)$ is its matroid retraction.
\end{remark}

\Cref{lem:intersection-left-bruhat} is an intermediate step in understanding the set 
\[\mathcal{LR}(x,y)\coloneqq \{w\in W\mid w\leq_L x\text{ and }w\leq_Ry\}
\]
of interest in \Cref{thm:left-right-max}, which states that $\mathcal{LR}(x,y)$ has a unique maximal element $\lr(x,y)$ under Bruhat order.

\begin{proof}[Proof of \Cref{thm:left-right-max}]
Construct a sequence of elements $\{z_i\}_{i\geq0}$ as follows. We start with $z_0=x$. Then using \Cref{lem:intersection-left-bruhat}, for $i\geq0$, define \[z_{2i+1}=\br(z_{2i},y),\quad z_{2i+2}=\lb(x,z_{2i+1}).\]
By construction, $x=z_0\geq_B z_1\geq_B z_2\geq_B\cdots$, so since $[e,x]_B$ is finite, this sequence stabilizes at some $z\eqqcolon \lr(x,y)$. Then $z=\br(z,y)\leq_R y$ and $z=\lb(x,z)\leq_L x$, so $z\in\mathcal{LR}(x,y)$. 

Now take any $w\in\mathcal{LR}(x,y)$; we argue by induction that $w\leq_B z_i$ for all $i$. The base case $w\leq_B z_0=x$ is clear as $w\leq_L x$. Suppose that $w\leq_B z_{2i}$. Then $w\in \mathcal{BR}(z_{2i},y)$, which means that $w\leq_B \br(z_{2i},y)=z_{2i+1}$ by \Cref{lem:intersection-left-bruhat}. Dually, if $w\leq_B z_{2i+1}$, then $w\in\mathcal{LB}(x,z_{2i+1})$ and thus $w\leq_B \lb(x,z_{2i+1})=z_{2i+2}$. Thus $w\leq_B z$ and we are done.
\end{proof}

\begin{proof}[Proof of \Cref{cor:Hart-Rowley}]
Let $a(w)=\lr(w^{-1},w)$. It is clear that $\lr(x,y)=\lr(y^{-1},x^{-1})^{-1}$ for $x,y\in W$, so $a(w)=a(w)^{-1}$. Since $w\neq e$, there is $s\in S$ with $s\leq_R w$; then $s\leq_L w^{-1}$, so $s\in\mathcal{LR}(w^{-1},w)$ and $s\leq_B a(w)$ by \Cref{thm:left-right-max}. Hence $a(w)\neq e$, meaning that $a(w)$ is an involution. Moreover, $a(w)\leq_R w$ since $a(w)\in\mathcal{LR}(w^{-1},w)$. Every involution $\iota\leq_R w$ satisfies $\iota=\iota^{-1}\leq_L w^{-1}$, so $\iota \in\mathcal{LR}(w^{-1},w)$. By \Cref{thm:left-right-max}, $\iota \leq_B a(w)$. Since Bruhat order is graded by length, this gives $\ell(\iota)\leq\ell(a(w))$, with equality only if $\iota=a(w)$; thus $a(w)$ is the unique involution of maximal length in $[e,w]_R$. 
\end{proof}

Determining when $\mathcal{LR}(x,y)=\{e\}$ is a central focus of the subsequent sections.

\section{The surjection theorem}
\label{sec:surjection}
The type $A$ case of \Cref{thm:surjection} was previously obtained via an intricate analysis of wiring diagrams \cite{separable-2}. The technology developed in \Cref{sec:left-right} allows for a proof in all Coxeter groups.

\begin{lem}\label{lem:overlap-bound}
Let $x,y,w\in W$ be such that $x\leq_R w$ and $y\leq_L w$. Then there exists
$z\in\mathcal{LR}(x,y)$ such that
$\ell(z)\geq\ell(x)+\ell(y)-\ell(w)$.
\end{lem}
\begin{proof}
Write $w=xa=by$ with both factorizations reduced, and induct on $\ell(a)$. If $a=e$, then $y\leq_L w=x$, so we can take $z=y$. Now suppose $\ell(a)\geq1$ and take $a=a's$ reduced, where $s\in S$. Put $w'=ws=xa'$, so that $\ell(w')=\ell(w)-1$. If $\ell(ys)<\ell(y)$, put $y'=ys$. The factorization $w'=by'$ is reduced, and the induction hypothesis applied to $w'=xa'=by'$ gives $z\leq_L x$ and $z\leq_R y'\leq_R y$ with $\ell(z)\geq\ell(x)+\ell(y')-\ell(w')=\ell(x)+\ell(y)-\ell(w)$, as desired.

Now suppose $\ell(ys)>\ell(y)$. Fix reduced expressions for $b$ and $y$, whose concatenation is a reduced expression for $w$. By the exchange condition, deleting one letter from this concatenation gives a reduced expression for $w'=ws$. The deleted letter cannot lie in the expression for $y$ since otherwise canceling $b$ would express $ys$ by a word of length $\ell(y)-1$, contradicting $\ell(ys)=\ell(y)+1$. Thus the deletion occurs in the expression for $b$. Let $b'$ be the element represented by the remaining letters of that expression. Then $w'=b'y$ is reduced, since the concatenated word has $\ell(w)-1=\ell(w')$ letters. Applying the induction hypothesis to $w'=xa'=b'y$ gives $z\in\mathcal{LR}(x,y)$ with $\ell(z)\geq\ell(x)+\ell(y)-\ell(w')=\ell(x)+\ell(y)-\ell(w)+1$, which is stronger than required.
\end{proof}

The following is the main technical lemma that we need.
\begin{lem}\label{lem:squeeze-if-not-reduced}
Let $x,y\in W$ and $z\in\mathcal{LR}(x,y)$. If \[\ell(xz^{-1})+\ell(z)+\ell(z^{-1}y)>\ell(xz^{-1}y),\]
then there exists $z'\in \mathcal{LR}(x,y)$ such that $z<_B z'$.
\end{lem}
\begin{proof}
Write $x=az$ and $y=zb$, with both products reduced. The hypothesis says that $azb$ is not reduced. Fix a reduced expression $b=s_1\cdots s_m$, and let $j$ be the first index for which $azs_1\cdots s_j$ is not reduced. Put $b_0=s_1\cdots s_{j-1}$, $s=s_j$, $w=azb_0=xb_0$, and $y_0=zb_0s$. By construction, $w=xb_0$ is reduced and $\ell(ws)=\ell(w)-1$. Also, $y_0$ is a reduced prefix of $y$, so $y_0\leq_R y$ and $\ell(y_0)=\ell(z)+\ell(b_0)+1$.

Choose reduced expressions for $a$, $z$, and $b_0$, whose concatenation is a reduced expression for $w$. By the exchange condition, deleting one letter from this concatenation gives a reduced expression for $ws$. The deleted letter must lie in the expression for $a$. Otherwise, canceling $a$ from the resulting equality would express $y_0=zb_0s$ by a word of length $\ell(z)+\ell(b_0)-1$, contradicting the preceding length formula. Let $a'$ be the element represented by the remaining letters of the expression for $a$. Then $ws=a'zb_0$ is reduced and $\ell(a')=\ell(a)-1$. Multiplying by $s$ gives $w=a'y_0$, and this factorization is also reduced, since $\ell(a')+\ell(y_0)=\ell(w)$.

We therefore have two reduced factorizations $w=xb_0=a'y_0$. \Cref{lem:overlap-bound} gives an element $v\in\mathcal{LR}(x,y_0)$ satisfying $\ell(v)\geq\ell(x)+\ell(y_0)-\ell(w)=\ell(z)+1$. Since $y_0\leq_R y$, we also have $v\in\mathcal{LR}(x,y)$. Finally, take $z'=\lr(x,y)$. By \Cref{thm:left-right-max}, both $z\leq_B z'$ and $v\leq_B z'$. Thus $\ell(z')\geq\ell(v)>\ell(z)$, so $z<_B z'$, as required.
\end{proof}

\begin{proof}[Proof of \Cref{thm:surjection}]
Fix $w\in W$ and let $z=\lr(w,u)$, which exists by \Cref{thm:left-right-max}. Write $w=az$ and $u=zb$, with both factorizations reduced. If $azb$ were not reduced, \Cref{lem:squeeze-if-not-reduced} would give $z'\in\mathcal{LR}(w,u)$ with $z<_B z'$, contradicting the maximality of $z$. Therefore $au=azb$ is reduced, so $\ell(au)=\ell(a)+\ell(u)$.

For any $v\in U$, write $u=vc$ reduced. Then $\ell(avc)=\ell(a)+\ell(v)+\ell(c)$, so $av$ is also reduced. Hence $a\in W/U$. Since $z\in U$ and $w=az$, the pair $(a,z)$ maps to $w$, proving surjectivity.
\end{proof}

\section{The non-separable cases}
\label{sec:non-sep}
In this section, we restrict to finite Weyl groups, where separable elements are defined \cite{separable-1}. Here we have $W/U=[e,w_0u^{-1}]_L$ for $U=[e,u]_R$ \cite{Bjorner-Wachs}. 

Our main goal, \Cref{thm:separable-bijection}, has already been established in type $A$ \cite{separable-2} and type $B$ \cite{separable-typeB}. Type $C$ follows from type $B$: the root systems of types $B_n$ and $C_n$ are dual to each other, and the bijection $\beta\mapsto\beta^\vee=2\beta/\langle\beta,\beta\rangle$ is $W$-equivariant, sends simple roots to simple coroots, and preserves positivity and supports. It therefore identifies inversion sets, restrictions to standard parabolic subsystems, and pivots, so the separable elements of $W(B_n)$ and $W(C_n)$ coincide. The important remaining cases are types $D$ and $E$, as the remaining exceptional types are easy to deal with computationally. Our discussion of types $D$ and $E$ below is uniform and graph-theoretic.

\subsection{Equivalent conditions for splitting}
The following result will be used in the proof of \Cref{cor:separable-iff-LR-e} to link separable elements to splittings and the operation $\lr$.

\begin{thm}\label{thm:bijection-iff-LRempty}
Let $xy=w_0$ and $U=[e,y]_R$, so that $W/U=[e,x]_L$. The multiplication map $W/U\times U\rightarrow W$ is a bijection if and only if $\lr(x,y)=e$.
\end{thm}
\begin{proof}
Write $z=\lr(x,y)$. Since $z\in[e,x]_L\cap[e,y]_R=(W/U)\cap U$, the element $z=e\cdot z=z\cdot e$ is the image of both $(e,z)$ and $(z,e)$; so if $W/U\times U\rightarrow W$ is a bijection, then $z=e$. For the more interesting direction, suppose that $W/U\times U\rightarrow W$ is not a bijection. By \Cref{thm:surjection}, this map is surjective, so there is some $w$ and distinct pairs $(a,b),(a',b')\in W/U\times U$ such that $w=ab=a'b'$. Both factorizations are reduced, so $b,b'\in\mathcal{LR}(w,y)$.

Let $r=\lr(w,y)$. By \Cref{thm:left-right-max}, both $b\leq_B r$ and $b'\leq_B r$. Since $b\neq b'$ (as $b=b'$ would force $a=wb^{-1}=a'$), at least one inequality is strict; without loss of generality, $b<_B r$. Now $a\leq_R w$ and $r\leq_L w$, so \Cref{lem:overlap-bound} gives $v\in\mathcal{LR}(a,r)$ with $\ell(v)\geq\ell(a)+\ell(r)-\ell(w)=\ell(r)-\ell(b)>0$. Moreover, $v\leq_L a\leq_L x$ and $v\leq_R r\leq_R y$, so $v\in\mathcal{LR}(x,y)$. Therefore $v\leq_B z$ and $z\neq e$.
\end{proof}

\begin{lem}\label{lem:restriction-w0}
Let $xy=w_0$ and $K\subseteq S$. Then $x_K\cdot{}_K y=w_0(K)$.
\end{lem}
\begin{proof}
The factorization $w_0=x^K\cdot x_K\cdot{}_K y\cdot{}^K y$ is reduced and, in particular, $x_K\cdot{}_K y$ is a reduced product in $W(K)$. For every $s\in K$, $\ell(x^Ks)=\ell(x^K)+1$. Then $\ell(sx_Ky)=\ell(w_0)-\ell(x^Ks)=\ell(x_Ky)-1$. On the other hand, the left parabolic decomposition of $sx_K y$ with respect to $K$ is $(sx_K\cdot{}_Ky)({}^K y)$, so $\ell(sx_Ky)=\ell(sx_K\cdot {}_K y)+\ell({}^K y)$. Comparing with $\ell(x_Ky)=\ell(x_K\cdot{}_K y)+\ell({}^K y)$, we obtain $\ell(sx_K\cdot{}_K y)=\ell(x_K\cdot{}_K y)-1$. Thus every $s\in K$ is a left descent of $x_K\cdot{}_K y$, forcing $x_K\cdot{}_K y=w_0(K)$.
\end{proof}

\subsection{Simply-laced types}\label{sub:simply-laced}

In a finite simply-laced Weyl group, we can choose the inner product on the ambient vector space so that $\langle\beta,\beta\rangle=2$ for $\beta\in\Phi^+$, and $\langle\beta,\gamma\rangle\in\{0,1,-1\}$ for $\beta\neq\gamma\in\Phi^+$. Moreover, for $\beta,\gamma\in\Phi^+$, if $\gamma-\beta\in\Phi^+$, then $t_{\beta}\gamma=\gamma-\beta$, and if $\gamma+\beta\in\Phi^+$, then $t_{\beta}\gamma=t_{\gamma}\beta=\gamma+\beta$.

\begin{lem}\label{lem:type-DE-main}
Let $(W,S)$ be a Coxeter system of type $D$ or $E$, and let $xy=w_0$, where $x$ is not separable (or equivalently, $y$ is not separable) but $x_J$ is separable (or equivalently, ${}_Jy$ is separable, by \Cref{lem:restriction-w0}) for every $J\subsetneq S$. Then there exists a \emph{witness} $z\in\mathcal{LR}(x,y)\setminus\{e\}$.
\end{lem}
\begin{proof}
For simplicity of notation, we call a positive root in $I_R(x)$ red and a positive root in $I_L(y)$ blue. Note that $I_R(x)\sqcup I_L(y)=\Phi^+$, so the set $\Phi^+$ of all positive roots admits a coloring into red and blue. The goal is to explicitly construct $z$ such that all roots in $I_R(z)$ are red and all roots in $I_L(z)$ are blue.

For a connected subgraph $H$ of the Dynkin diagram $G$, write $\gamma(H)\coloneqq \sum_{s\in H} \alpha_s$, which is indeed a positive root. For $s,s'\in S$, we also write $\gamma(s,s')$ for the \emph{path root} that is the sum of the simple roots $\alpha_u$ over all vertices $u$ in the unique path connecting $s$ and $s'$. These roots, whose coefficients all lie in $\{0,1\}$, are called \emph{small roots} in \cite{separable-1} and are of great importance in establishing the pattern-avoidance criterion for separable elements.

\begin{claim}\label{claim:witness-main}
Let $\alpha_s$ be red and choose a blue root $\beta$ of minimum height among roots containing $s$. Then $[\alpha_s]\beta=1$. Thus, a saturated chain from $\alpha_s$ to $\beta$ in the root poset gives $w\in W(S\setminus\{s\})$ such that $w\alpha_s=\beta$ and $\ell(w)=\operatorname{ht}(\beta)-1$. Set $z=ws$. Either $z\in\mathcal{LR}(x,y)$ is a witness, or $\beta=\gamma(s,t)$ is a path root with endpoints $s$ and $t$, where $\alpha_t$ is blue.
\end{claim}
\begin{proof}
Since $x_{S\setminus\{s\}}$ is separable and $x$ is not, the roots containing $s$ cannot all be red. Thus such a $\beta$ exists. Put $K=S\setminus\{s\}$ and $h=\operatorname{ht}(\beta)$.

Suppose first that $[\alpha_s]\beta=k\geq2$, and choose a root $\beta_0$ of minimum height in the $W(K)$-orbit of $\beta$. Every root in this orbit is positive, since its coefficient at $\alpha_s$ remains $k$. Minimality gives $\langle\beta_0,\alpha_t\rangle\leq0$ for every $t\in K$, and hence $2=\langle\beta_0,\beta_0\rangle\leq k\langle\beta_0,\alpha_s\rangle$. Since $\beta_0\neq\alpha_s$, it follows that $\langle\beta_0,\alpha_s\rangle=1$, so $\beta_0-\alpha_s$ is a positive root. Choose $g\in W(K)$ with $g\beta_0=\beta$. The decomposition $\beta=g\alpha_s+g(\beta_0-\alpha_s)$ expresses $\beta$ as a sum of two positive roots containing $s$: their coefficients at $\alpha_s$ are $1$ and $k-1$, respectively. Both have smaller height than $\beta$, so both are red, contradicting biconvexity. Therefore $[\alpha_s]\beta=1$.

If a positive root $\eta\neq\alpha_s$ has coefficient one at $\alpha_s$, then some $t\in K$ satisfies $\langle\eta,\alpha_t\rangle>0$. Otherwise, $2=\langle\eta,\eta\rangle\leq\langle\eta,\alpha_s\rangle\leq1$, a contradiction. Thus $\eta$ can be lowered to $\eta-\alpha_t$ without changing its coefficient at $\alpha_s$. Repeating gives a saturated lowering chain from $\beta$ to $\alpha_s$. Let $w$ be the product of its reflection labels in lowering order, so that $w\alpha_s=\beta$ and $w\in W(K)$. This word has $h-1$ letters, whose corresponding simple roots sum to $\beta-\alpha_s$. Every word in $K$ taking $\alpha_s$ to $\beta$ has at least $h-1$ letters, so $\ell(w)=h-1$.

We next show that every root in $I_L(z)=I_L(w)\sqcup\{\beta\}$ is blue. Take $\gamma\in I_L(w)$. Then $t_\gamma w<_B w$, while $t_\gamma\beta$ is still a positive root with coefficient one at $\alpha_s$, since $\gamma$ is supported in $K$. Since $t_{\gamma}w$ takes $\alpha_s$ to $t_{\gamma}\beta$, we have $\operatorname{ht}(t_\gamma\beta)\leq\ell(t_\gamma w)+1<h$. Hence $\langle\beta,\gamma\rangle=1$, and $\beta-\gamma$ is a positive root containing $s$ of smaller height than $\beta$. It is therefore red. Since $\beta=(\beta-\gamma)+\gamma$ is blue, biconvexity forces $\gamma$ to be blue as well. Thus $I_L(z)$ consists entirely of blue roots.

We then analyze $I_R(z)=\{\alpha_s\}\sqcup\{s\gamma\mid \gamma\in I_R(w)\}$. For $\gamma\in I_R(w)$, we have $\langle\alpha_s,\gamma\rangle\leq0$, because $\gamma$ is supported in $K$. We cannot have $\langle\alpha_s,\gamma\rangle=0$, since otherwise $wt_\gamma$ would be a shorter element of $W(K)$ taking $\alpha_s$ to $\beta$. Therefore $\langle\alpha_s,\gamma\rangle=-1$, giving $s\gamma=\alpha_s+\gamma$. This is a root containing $s$, so as long as $\operatorname{ht}(\gamma)\leq h-2$, the root $\alpha_s+\gamma$ is red by the minimality of $\beta$. The only scenario in which $\alpha_s+\gamma$ can be blue is that any reduced word of $w$, which has length $h-1$, lowers the height of $\gamma$ one step at a time before a final step that changes its sign. It follows that $\gamma$ is the sum of the simple roots labeling that word, namely $\beta-\alpha_s$. Thus the only possible blue root in $I_R(z)$ is $\beta$ itself.

In summary, unless $\beta-\alpha_s\in I_R(w)$, we have constructed a witness.

It remains to show that this exceptional case forces $\beta$ to be a path root with endpoint $s$. We will prove that if $r\in S$, $\eta\in\Phi^+$, and $v\in W(S\setminus\{r\})$ satisfy \[v\alpha_r=\eta,\quad\ell(v)=\operatorname{ht}(\eta)-1,\quad \eta-\alpha_r\in I_R(v),\]
then $\eta=\gamma(r,t)$ is a path root for some $t\neq r$ with $t\in D_L(v)$. Induct on $\operatorname{ht}(\eta)\geq2$. If $\operatorname{ht}(\eta)=2$, then $\eta=\alpha_r+\alpha_{t}=\gamma(r,t)$ with $t$ adjacent to $r$ in the Dynkin diagram, and $v=t$. For the general case, a reduced word of $v$ acting on $\alpha_r$ must increase its height one step at a time. So $\eta-\alpha_r$ is the sum of the simple roots labeling any reduced expression of $v$, counted with multiplicity. Moreover, as $v\in W(S\setminus\{r\})$ and $\eta-\alpha_r$ is a root, there exists a unique neighbor $r'$ of $r$ in $\Supp(\eta-\alpha_r)$, with coefficient $1$. Also, $\Supp(\eta-\alpha_r)$ lies in a single connected component of $S\setminus\{r\}$. This also shows that any reduced word of $v$ ends with $r'$. Write $v=v'r'$. Since $r'$ has coefficient $1$ in $\eta-\alpha_r$, the letter $r'$ appears exactly once in any reduced word of $v$, so $\Supp(v')\subseteq\Supp(\eta-\alpha_r)\setminus\{r'\}\subseteq S\setminus\{r,r'\}$. In particular, $\Supp(v')$ contains no neighbor of $r$, so $v'\alpha_r=\alpha_r$ and hence $v'\alpha_{r'}=v'(r'\alpha_r-\alpha_r)=v\alpha_r-\alpha_r=\eta-\alpha_r$. Since $v'r'(\eta-\alpha_r)<0$, every step from any reduced word of $v'$ lowers the height by $1$, until the final step changes sign. In particular, this means $r'(\eta-\alpha_r)=\eta-\alpha_r-\alpha_{r'}\in I_R(v')$. The triple $(r',\eta-\alpha_r,v')$ satisfies the induction hypothesis, which means $\eta-\alpha_r=\gamma(r',t)$ for some $t\in D_L(v')$; note that $D_L(v')\subseteq D_L(v)$ since $v'\leq_R v$. As $\Supp(\eta-\alpha_r)$ does not contain $r$, we necessarily have $\eta=\alpha_r+\gamma(r',t)=\gamma(r,t)$ as desired. 

Finally, since $\beta$ is a blue root of minimum height containing $s$, the root $\beta-\alpha_t$ must be red, forcing $\alpha_t$ to be blue by biconvexity. 
\end{proof}

By the \emph{color-reversed version} of \Cref{claim:witness-main}, we mean \Cref{claim:witness-main} applied to the pair $(y^{-1},x^{-1})$ in place of $(x,y)$. This pair also satisfies the hypotheses of \Cref{lem:type-DE-main}: we have $y^{-1}x^{-1}=w_0$, the element $y^{-1}$ is not separable, and $(y^{-1})_J=({}_Jy)^{-1}$ is separable for every $J\subsetneq S$, since separability is preserved by inversion. Since $I_R(y^{-1})=I_L(y)$ and $I_L(x^{-1})=I_R(x)$, passing to this pair swaps red and blue. Moreover, $\mathcal{LR}(y^{-1},x^{-1})=\mathcal{LR}(x,y)^{-1}$, so a witness for one pair inverts to a witness for the other.

Now suppose, for the sake of contradiction, that $\mathcal{LR}(x,y)=\{e\}$. \Cref{claim:witness-main} and its color-reversed version imply that every minimum-height root of the opposite color containing a vertex $u$ is a path root $\gamma(u,v)$, where $v$ has that opposite color.

Let $o$ be the trivalent vertex of the Dynkin diagram $G$. For each leaf $l$, choose a pivot $p_l$ of the separable restriction obtained by deleting $l$. Every positive root containing $p_l$ but avoiding $l$ has the color of $p_l$. A minimum-height root of the opposite color containing $p_l$ must contain $l$. By \Cref{claim:witness-main} and its color-reversed version, this root is a path root with endpoint $p_l$. Its other endpoint must be $l$, since $l$ is a leaf. Consequently, $p_l$ and $l$ have opposite colors, and $\gamma(p_l,l)$ has the color of $l$.

\begin{claim}\label{claim:leaves-not-monochromatic}
The three leaves of $G$ cannot all have the same color.
\end{claim}
\begin{proof}
Suppose, for the sake of contradiction, that all three leaves have the same color; by symmetry, we may assume that they are all blue. Then their corresponding pivots are red. The three path roots joining a leaf to its corresponding pivot are all blue, so these three paths do not contain other pivots and are thus disjoint. Since we are working in type $D$ or $E$, assume without loss of generality that the leaf $a$ is adjacent to $o$, so that the path from $a$ to $p_a$ contains $o$. Name the other two leaves $b$ and $c$ so that $p_a$ lies on the path from $o$ to $b$, possibly with $p_a=o$. Disjointness forces $p_c$ to lie on the arm from $o$ to $c$. See \Cref{fig:Claim2-cartoon}.
\begin{figure}[h!]
\centering
\begin{tikzpicture}[scale=0.8]
\node at (0,0) {$\bullet$};
\node[below,text=blue] at (0,0) {$b$};
\node at (1,0) {$\bullet$};
\node at (2,0) {$\bullet$};
\node[below,text=red] at (2,0) {$p_b$};
\node at (3,0) {$\bullet$};
\node[below,text=red] at (3,0) {$p_a$};
\node at (4,0) {$\bullet$};
\node[below] at (4,0) {$o$};
\node at (5,0) {$\bullet$};
\node[below,text=red] at (5,0) {$p_c$};
\node at (6,0) {$\bullet$};
\node[below,text=blue] at (6,0) {$c$};
\node at (4,1) {$\bullet$};
\node[right,text=blue] at (4,1) {$a$};
\draw(0,0)--(6,0);
\draw(4,0)--(4,1);
\end{tikzpicture}
\caption{A visualization for \Cref{claim:leaves-not-monochromatic}.}
\label{fig:Claim2-cartoon}
\end{figure}

Let $H$ be the subtree spanned by $a$, $c$, and $p_a$; it also contains $p_c$. Since $b\notin H$, the restriction of $x$ to $H$ is separable; choose a pivot $u$ of it. As $\gamma(H\setminus\{a\})$ contains $p_a$ and avoids $a$, it is red, and similarly $\gamma(H\setminus\{c\})$ is red as well. The supports of these two roots cover $H$, meaning that $u$ must be red. Choose a minimum-height blue root containing $u$. By \Cref{claim:witness-main}, it is a path root $\gamma(u,v)$ with $v$ blue. This path cannot be contained in $H$, since $u$ is a red pivot there. The other endpoint $v$ must lie outside $H$, on the omitted part of the $b$-arm. The path from $v$ to $u$ passes through $p_a$ and avoids $a$, meaning that this path root is red, a contradiction.
\end{proof}

We now relabel the vertices of the Dynkin diagram $G$. See \Cref{fig:Claim3-cartoon} for visualization. Let the three arms attached to $o$ be $A=\{a_1,\ldots,a_l\}$, $B=\{b_1,\ldots,b_m\}$ and $C=\{c_1,\ldots,c_n\}$, where $a_1,b_1,c_1$ are connected to $o$ and the leaves are $a_l=a$, $b_m=b$, $c_n=c$. For simplicity of notation, also write $a_0=b_0=c_0=o$. We no longer need to use the fact that at least one of $l,m,n$ is $1$. By \Cref{claim:leaves-not-monochromatic}, we may assume without loss of generality that $a,b$ are blue and $c$ is red. Then $p_a,p_b$ are red and $p_c$ is blue; $\gamma(a,p_a),\gamma(b,p_b)$ are blue and $\gamma(c,p_c)$ is red. The path root $\gamma(p_a,p_c)$ would have both colors unless $p_a=c$ or $p_c=a$. Similarly, $p_b=c$ or $p_c=b$. Since $\gamma(a,p_a)$ is blue and $\gamma(a,p_b)$ is red, $p_a\neq p_b$. We may assume without loss of generality that $p_a=c$ and $p_c=b$. 

For every $u$ on the path between $a$ and $c=p_a$, the path root $\gamma(a,u)$ is blue. This is true if $u=a$ or $u=c$. For $u$ strictly between $a$ and $c$, $\gamma(a,c)-\gamma(a,u)$ is red as $c=p_a$ is a red pivot of $G\setminus\{a\}$, and $\gamma(a,c)$ is blue, which means $\gamma(a,u)$ is blue by biconvexity. Applying the color-reversed version of \Cref{claim:witness-main} to $a$, we see that a minimum-height red root containing $a$ must be $\gamma(a,b_k)$ for some $1\leq k<m$. Thus $b_k$ is red.

\begin{figure}[h!]
\centering
\begin{tikzpicture}[scale=0.7]
\node at (0,0) {$\bullet$};
\node at (1,0) {$\bullet$};
\node at (2,0) {$\bullet$};
\node at (3,0) {$\bullet$};
\node at (4,0) {$\bullet$};
\node at (5,0) {$\bullet$};
\node at (6,0) {$\bullet$};
\node at (7,0) {$\bullet$};
\node at (8,0) {$\bullet$};
\node at (9,0) {$\bullet$};
\node at (5,2) {$\bullet$};
\node at (5,1) {$\bullet$};
\node[below,text=blue] at (0,0) {${b=p_c}$};
\node[below,text=red] at (2,0) {$b_k$};
\node[right,text=blue] at (5,2) {$a$};
\node[below,text=red] at (9,0) {$c=p_a$};
\draw(0,0)--(9,0);
\draw(5,0)--(5,2);
\end{tikzpicture}
\caption{A visualization for \Cref{claim:typeDE-final-witness}.}
\label{fig:Claim3-cartoon}
\end{figure}

\begin{claim}\label{claim:typeDE-final-witness}
Under the assumptions above, identifying vertices with their corresponding simple reflections, the element
\[
z=(a_l\cdots a_1)(b_m\cdots b_1)\,o\,(b_1\cdots b_k)(c_1\cdots c_n)
\]
is a nonidentity witness: every root in $I_R(z)$ is red and every root in $I_L(z)$ is blue.
\end{claim}
\begin{proof}
Put $\beta=\gamma(a,b)$ and $\delta=\gamma(c,b_k)$. For $1\leq i\leq l$ and $0\leq j\leq n$, write $A_i=\alpha_{a_1}+\cdots+\alpha_{a_i}$ and $C_j=\alpha_{c_1}+\cdots+\alpha_{c_j}$, with $C_0=0$.

Recall that for a reduced word $w=s_{i_1}s_{i_2}\cdots s_{i_{\ell}}$, $I_L(w)=\{\alpha_{i_1},s_{i_1}\alpha_{i_2},s_{i_1}s_{i_2}\alpha_{i_3},\ldots\}$. Thus, we calculate explicitly that
\begin{align*}
I_L(z)={}&
\{\gamma(a,a_i)\mid 1\leq i\leq l\}
\sqcup\{\gamma(b,b_j)\mid 1\leq j\leq m\}\\
&\sqcup\{\gamma(a,b_j)\mid 0\leq j<k\}
\sqcup\{\beta+C_i\mid 0\leq i\leq n\},\\
I_R(z)={}&
\{\gamma(c,u)\mid u\in B\cup C\cup\{o\}\}\\
&\sqcup\{\gamma(b_j,b_k)\mid 1\leq j\leq k\}
\sqcup\{\delta+A_i\mid 1\leq i\leq l\}.
\end{align*}

We check the colors of the left inversions first. Since $\gamma(a,b_k)$ is a minimum-height red root containing $a$, every $\gamma(a,a_i)$ and every $\gamma(a,b_j)$ with $j<k$ is blue. Each $\gamma(b,b_j)$, as well as each $\beta+C_i$ with $i<n$, contains the blue pivot $p_c=b$ and avoids $c$, so these roots are blue as well. Finally,
$\beta+C_n=\gamma(a,c)+\gamma(b_1,b)$ is blue by biconvexity: $\gamma(a,c)$ is blue by the preceding argument, and $\gamma(b_1,b)$ contains $p_c=b$ and avoids $c$. Hence every root in $I_L(z)$ is blue.

For the right inversions, each $\gamma(c,u)$ with $u\in B\cup C\cup\{o\}$ contains the red pivot $p_a=c$ and avoids $a$, so it is red. For $1\leq j\leq k$, the decomposition
$\gamma(a,b_k)=\gamma(a,b_{j-1})+\gamma(b_j,b_k)$ has a red left-hand side and a blue first summand. Biconvexity therefore forces $\gamma(b_j,b_k)$ to be red. Similarly, $\delta+A_i$ contains $p_a=c$ and avoids $a$ whenever $i<l$, so it is red. The final root satisfies
$\delta+A_l=\gamma(a,b_k)+C_n$, a sum of two red roots, and is therefore red.
\end{proof}

The preceding claims produce an element $z\in\mathcal{LR}(x,y)\setminus\{e\}$, contradicting our assumption that $\mathcal{LR}(x,y)=\{e\}$.
\end{proof}

\subsection{Finishing the proof of \texorpdfstring{\Cref{thm:separable-bijection}}{Theorem 1.4}}
\begin{proof}[Proof of \Cref{thm:separable-bijection}]
If $y$ is separable, then the multiplication map $W/U\times U\rightarrow W$ is a length-additive bijection, where $U=[e,y]_R$ and $W/U=[e,x]_L$ for $xy=w_0$, by \cite[Thm.~3]{separable-2}. Now assume that $xy=w_0$ and that neither $x$ nor $y$ is separable. It suffices to find two distinct pairs in $W/U\times U$ with the same product; for instance, a witness $z\in\mathcal{LR}(x,y)\setminus\{e\}$ gives the two pairs $(e,z)$ and $(z,e)$.

Choose $J\subseteq S$ minimal such that $x_J$ is not separable. For every $J'\subsetneq J$, the element $(x_J)_{J'}=x_{J'}$ is separable. By (S1) and (S2), $J$ is connected and has at least two elements. By \Cref{lem:restriction-w0}, $x_J\cdot{}_Jy=w_0(J)$, so $W(J)/[e,{}_Jy]_R=[e,x_J]_L$. Moreover, $[e,x_J]_L\subseteq[e,x]_L=W/U$ and $[e,{}_Jy]_R\subseteq[e,y]_R=U$. Hence two distinct pairs in $[e,x_J]_L\times[e,{}_Jy]_R$ with the same product, and in particular a witness in $\mathcal{LR}(x_J,{}_Jy)$, give two distinct pairs in $W/U\times U$ with the same product. Replacing $(W,x,y)$ by $(W(J),x_J,{}_Jy)$, we may therefore assume that $W$ is irreducible and that for any $J\subsetneq S$, $x_J$ and ${}_Jy$ are separable.

If $(W,S)$ is of type $D$ or $E$, \Cref{lem:type-DE-main} produces such a witness. If $(W,S)$ is of exceptional type $F_4$ or $G_2$, a quick computational check also generates a witness for each non-separable pair. If $(W,S)$ is of type $A$, we previously established the theorem in \cite[Thm.~4]{separable-2}, while type $B$ was proved by Liu and Yu \cite[Thm.~1.2]{separable-typeB}; type $C$ follows from type $B$, since the two types have the same Weyl group and the same separable elements (see the beginning of this section). In these cases the cited results show that the multiplication map is not a bijection, hence not injective by \Cref{thm:surjection}, which suffices by the previous paragraph. The arguments in previous papers for types $A$ and $B$ are significantly different from our analysis here. In fact, continuing our line of thought yields much shorter proofs in types $A$ and $B$ by describing a witness explicitly.

For type $A$, suppose first that $x$ contains the pattern $2413$. Then a witness is given by $z=(s_1s_2\cdots s_{n-1-x(1)})(s_{n-1}s_{n-2}\cdots s_{n+1-x(n)})$. And if $x$ contains $3142$ instead, then $y^{-1}=w_0x$ contains $2413$. Since $y^{-1}x^{-1}=w_0$ and $\mathcal{LR}(y^{-1},x^{-1})=\mathcal{LR}(x,y)^{-1}$, applying the construction to the pair $(y^{-1},x^{-1})$ and inverting the resulting witness gives a witness for $(x,y)$.

For type $B$, adopt the convention of signed permutations (see \cite{separable-typeB} for details) that $S=\{s_0,s_1,\ldots,s_{n-1}\}$, where $s_0$ corresponds to the simple root $\mathbf{e}_1$ and $s_i$ corresponds to the simple root $\mathbf{e}_{i+1}-\mathbf{e}_i$ for $i\in[n-1]$. The Weyl group $W$ consists of the permutations $w$ of $\{-n,\ldots,-1,1,\ldots,n\}$ such that $w(-i)=-w(i)$. Normalize so that $x(n)>0$. If $x(n)<0$, then $y^{-1}=w_0x$ satisfies $y^{-1}(n)>0$, so apply the construction to the pair $(y^{-1},x^{-1})$ and invert the resulting witness, as in type $A$.

If $x(1)<x(n)$, let $h=\min\{i<n\mid x(i)<-x(n)\}$ and $k=\max\{i<n\mid x(i)<0\}$. Then
\[
z=(s_ks_{k+1}\cdots s_{n-1})(s_{k-1}s_{k-2}\cdots s_0)(s_1s_2\cdots s_{h-1})
\]
is a witness. If $x(1)>x(n)$, let $h=\max\{i<n\mid x(i)>x(n)\}$ and $k=\max\{i<n\mid x(i)>0\}$; then we have a witness
\[
z=(s_{n-1}s_{n-2}\cdots s_{k+1})(s_0s_1\cdots s_k)(s_{k-1}s_{k-2}\cdots s_h).
\]
\end{proof}

\begin{proof}[Proof of \Cref{cor:separable-iff-LR-e}]
Let $W$ be a finite Weyl group and $u \in W$. Then for $U=[e,u]_R$ we have $W/U=[e,w_0u^{-1}]_L$ by \cite{Bjorner-Wachs}. By \Cref{thm:separable-bijection} we have that $(W/U,U)$ is a splitting of $W$ if and only if $u$ is separable. On the other hand, by \Cref{thm:bijection-iff-LRempty} this holds if and only if $\lr(w_0u^{-1},u)=e$.
\end{proof}

\subsection*{Acknowledgments}
C.G. was partially supported by National Science Foundation award no. DMS-2452032 and by a travel grant from the Simons Foundation. Y.G. was partially supported by NSFC Grant no. 12471309. We are grateful to Grant Barkley, Christophe Hohlweg, and Oded Yacobi for helpful conversations.

\subsection*{Statement on AI use} The authors long conjectured that \Cref{thm:left-right-max} and \Cref{lem:squeeze-if-not-reduced} were true and were key to proving \Cref{thm:surjection}. ChatGPT produced the proof of \Cref{thm:left-right-max} and was also used for some details of several of the technical lemmas.

\bibliographystyle{halpha-abbrv}
\bibliography{separable3}
\end{document}